\documentclass[12pt]{article}

\usepackage[T1]{fontenc}
\usepackage[latin9]{inputenc}
\usepackage{color}
\usepackage{float}
\usepackage{mathtools}
\usepackage{amsmath}
\usepackage{amsthm}
\usepackage{amssymb}
\usepackage{graphicx}
\usepackage[unicode=true,pdfusetitle,
 bookmarks=true,bookmarksnumbered=true,bookmarksopen=false,
 breaklinks=false,pdfborder={0 0 0},pdfborderstyle={},backref=false,colorlinks=true]
 {hyperref}
\hypersetup{
 linkcolor=blue, urlcolor=blue, citecolor=blue, pdfstartview={FitH}, hyperfootnotes=false, unicode=true}

\makeatletter
\theoremstyle{plain}
\newtheorem{thm}{\protect\theoremname}
\theoremstyle{definition}
\newtheorem{defn}[thm]{\protect\definitionname}
\theoremstyle{remark}
\newtheorem{rem}[thm]{\protect\remarkname}
\theoremstyle{definition}
\newtheorem{example}[thm]{\protect\examplename}

\@ifundefined{date}{}{\date{}}
\usepackage{lettrine}
\let\myFoot\footnote
\renewcommand{\footnote}[1]{\myFoot{#1\vspace{3mm}}}
\usepackage{xcolor}
\usepackage{mlmodern}

\usepackage{tikz}
\usepackage{circuitikz}

\providecommand{\definitionname}{Definition}
\providecommand{\examplename}{Example}
\providecommand{\remarkname}{Remark}
\providecommand{\theoremname}{Theorem}
\providecommand{\theoremname}{Proposition}
\title{On the boundedness of Gross' solution when the underlying distribution is bounded}
\author{\footnote{
corresponding author, and e-mail:} 
\\$^1$Department of Mathematics and Statistics\\King Fahd University of Petroleum and Minerals\\Dhahran 31261, Saudi Arabia\\
 }

\makeatother

\providecommand{\definitionname}{Definition}
\providecommand{\examplename}{Example}
\providecommand{\remarkname}{Remark}
\providecommand{\theoremname}{Theorem}

\begin{document}
\title{On the boundedness of Gross' solution to the planar Skorokhod embedding
problem}
\author{Maher Boudabra, Dhaker Kroumi, Boubaker Smii \thanks{King Fahd University of Petroleum and Minerals }}
\maketitle
\begin{abstract}
In this work, we investigate the problem of the boundedness of the
Gross' solutions of the planar Skorokhod embedding problem, where
we show that the solution is bounded under some mild conditions on
the underlying probability distribution. 
\end{abstract}

\noindent \textbf{Mathematics Subject Classification (2010)}:

\section{Introduction}

In $2019$, the author R. Gross considered an interesting planar version
of the Skorokhod problem \cite{gross2019}, which was originally formulated
in $1961$ but in dimension one. For a concise survey of the one-dimensional
14 version, see \cite{Obloj2004}. The problem studied by Gross is
as follows : Let $\mu$ be a non-degenerate probability distribution
with zero mean and finite second moment. Is there a simply connected
domain $U$ (containing the origin) such that, for a $Z_{t}=X_{t}+Y_{t}i$
is a standard planar Brownian motion, then $X_{\tau}=\Re(Z_{\tau})$
has the distribution $\mu$? Here $\tau$ is the exit time of the
planar Brownian motion $Z_{t}$ from $U$. Gross provided an affirmative
answer, offering an explicit construction of his solution. In addition,
he showed that the underlying exit time $\tau$ has a finite average.
One year later, Boudabra and Markowsky published two papers on the
same problem \cite{boudabra2019remarks,Boudabra2020}. In the first
one, the authors demonstrated that the problem is solvable for any
non-degenerate distribution of a finite $p^{th}$moment where $p>1$.
Furthermore, they provided a uniqueness criterion. The second paper
provides a new category of domains that solve yet the Skorokhod embedding
problem as well as a uniqueness criterion. As in \cite{Boudabra2020},
we shall keep using the terminology $\mu$-domain to tag any simply
connected domain solving the planar Skorokhod problem. As this manuscript
deals with Gross' solution, we confine ourselves to it, that is, a
$\mu$-domain means simply constructed by Gross' technique. Let\textquoteright s
first summarize the geometric characteristics of the $\mu$-domains
generated by Gross\textquoteright{} method: : 
\begin{itemize}
\item $U$ is symmetric over the real line. 
\item $U$ is $\Delta$-convex, i.e the segment joining any point of $U$
and its symmetric point over the real axis remains inside $U$. 
\item If $\mu(\{x\})>0$ then $\partial U$ contains a vertical line segment,
a half line, or a line. 
\item If the support of the distribution $\mu$ has a gap from $a$ to $b$
then $U$ contains the vertical strip $(a,b)\times\mathbb{R}$. 
\end{itemize}
Note that the last two properties are universal, i.e they apply to
any potential solution of the planar Skorokhod embedding problem.
When it comes to boundedness, which is the focus of this note, any
$\mu$-domain $U$ is unbounded whenever the support of $\mu$ is
either unbounded or contains a gap. Specifically, $U$ will be horizontally
unbounded when the support of $\mu$ is unbounded, and vertically
unbounded if there is a gap within the support of $\mu$. Thus, two
necessary conditions for obtaining a bounded $\mu$-domain are the
support of $\mu$ must be both bounded and connected (without gaps).
Given these two assumptions, we will explore sufficient conditions
on $\mu$ that lead a bounded $\mu$-domain.

\section{Tools and Results}

We begin by outlining the ingredients of Gross' technique to generate
his $\mu$-domain, a solution to the planar Skorokhod embedding problem. 

The first one is the quantile function of $\mu$ defined by 
\begin{equation}
u\in(0,1)\mapsto q(u)=F^{-1}(u)=\inf\{x\mid F(x)\geq u\},\label{quantile-1-1}
\end{equation}
 where $F$ is the cumulative distribution function of $\mu$, i.e
$F(x)=\mu((-\infty,x])$, $x\in\mathbb{R}$. In other words, $q$
is the pseudo-inverse of $F$. When $F$ is increasing then $q$ simplifies
to the standard inverse function. A handy feature of $q$ is that,
when fed with uniformly distributed inputs in $(0,1)$, it generates
values sampling as $\mu$. Note that if $\mu$ has a gap, say $(a,b)$,
then $q$ jumps by $b-a$ at $u=F(a)$. The ``doubled periodic function''
$\phi$ is extracted out of $q$ by setting
\[
\theta\in(-\pi,\pi)\mapsto\phi(\theta)=q({\textstyle \frac{\vert\theta\vert}{\pi}}).
\]
Remark that the function $\phi$ is even and non-decreasing. 

The second ingredient is the periodic Hilbert transform, which will
control the range of the projection of the $\mu$-domain on the imaginary
axis. 
\begin{defn}
The Hilbert transform of a $2\pi$- periodic function $f$ is defined
by 
\[
H\{f\}(x):=PV\left\{ \frac{1}{2\pi}\int_{-\pi}^{\pi}f(x-t)\cot(\frac{t}{2})dt\right\} =\lim_{\eta\rightarrow0}\frac{1}{2\pi}\int_{\eta\leq|t|\leq\pi}f(x-t)\cot(\frac{t}{2})dt
\]
where $PV$ denotes the Cauchy principal value \cite{butzer1971hilbert}.
The role of $PV$ is to absorb infinite limits near singularities
in a certain sense. It is required for the Hilbert transform as the
trigonometric function $t\longmapsto\cot(\cdot)$ has poles at $k\pi$
with $k\in\mathbb{Z}$. The Hilbert transform is a bounded operator
on 
\[
L_{2\pi}^{p}=\{f:2\pi\text{-periodic function}\mid\int_{-\pi}^{\pi}\vert f(\theta)\vert^{p}d\theta<+\infty\}
\]
 for any $p>1$. More precisely, we have 
\end{defn}

\begin{thm}
\cite{butzer1971hilbert} If $f$ is in $L_{2\pi}^{p}$, then $H\{f\}$
exists almost everywhere for $p>1$. Furthermore, we have
\begin{equation}
\Vert H\{f\}\Vert_{p}\leq\lambda_{p}\Vert f\Vert_{p}\label{strong inequality}
\end{equation}
for some positive constant $\lambda_{p}$.
\end{thm}

The strong type estimate \ref{strong inequality} fails to hold when
$p=1$, as $H$ becomes unbounded. For further details see \cite{mcgovern1980hilbert,duren2000theory}
\cite{butzer1971hilbert,king2009hilbert}. 

Now we illustrate Gross' construction technique. He first generates
the Fourier series expansion of $\phi$: 
\[
\phi(\theta)=\sum_{n=1}^{+\infty}a_{n}\cos(n\theta),
\]
where $a_{n}$ is the $n^{th}$ Fourier coefficient of $\phi$. Note
that there is no constant term due the fact that $\mu$ is assumed
to be a centered probability distribution. Then he showed his cornerstone
result, upon which the solution is built. More precisely
\begin{thm}
\cite{gross2019} The analytic function 
\[
\Phi(z)=\sum_{n=1}^{+\infty}a_{n}z^{n}
\]
is univalent on the unit disc $\mathbb{D}$. 
\end{thm}

Using the conformal invariance principal of planar Brownian motion
\cite{morters2010brownian}, Gross shows that $\Phi(\mathbb{D})$,
i.e the image of the $\mathbb{D}$ under the action of $\Phi$, is
a solution for the Skorokhod embedding problem. If one knows that
$H\{\cos\}=\sin,$ then the boundary of his $\mu$-domain is parameterized
by 
\begin{equation}
\theta\in(-\pi,\pi)\mapsto(\phi(\theta),H\{\phi\}(\theta)).\label{eq:parametrized}
\end{equation}
For a Cartesian equation of \ref{eq:parametrized}, see \cite{Boudabra2020}. 

Now we state our first result. Let $\mu$ be a continuous probability
distribution concentrated on an interval $(a,b)$. Denote its density
by $\rho$ . In particular, the quantile function $q$ simplifies
to the standard inverse of $F$. We state now our first theorem. 
\begin{thm}
\label{bounded deriv theorem} If $\inf_{x\in(a,b)}\rho(x)$ is positive
then the underlying $\mu$-domain is bounded. 
\end{thm}

\begin{proof}
As $\rho$ is assumed to be positive then $\phi'$ is bounded on $[-\pi,\pi]$
since 
\[
\vert\phi'(\theta)\vert=\frac{1}{\pi}\frac{1}{\vert\rho(\phi(\theta))\vert}\leq\frac{1}{\pi}\frac{1}{\inf_{x\in(a,b)}\rho(x)}.
\]
Let $\theta$ be a fixed number in $(-\pi,\pi)$. The Hilbert transform
of $\phi$ is well defined as $\phi$ is bounded. By splitting the
integral in $H\{\phi\}$ into two parts, we have 
\begin{equation}
H\{\phi\}(\theta)=\frac{1}{2\pi}\lim_{\varepsilon\to0^{+}}\left\{ \int_{-\pi}^{-\varepsilon}\phi(\theta-t)\cot(\frac{t}{2})dt+\int_{\varepsilon}^{\pi}\phi(\theta-t)\cot(\frac{t}{2})dt\right\} .\label{eq1}
\end{equation}
Moreover, using a simple integration by parts, we obtain 
\begin{equation}
\int_{\varepsilon}^{\pi}\phi(\theta-t)\cot(\frac{t}{2})dt=-2\phi(\theta-\varepsilon)\ln(\vert\sin(\frac{\varepsilon}{2})\vert)+2\int_{\varepsilon}^{\pi}\phi'(\theta-t)\ln(\vert\sin(\frac{t}{2})\vert)dt.\label{eq2}
\end{equation}
Similarly, 
\begin{equation}
\begin{split}\int_{-\pi}^{-\varepsilon}\phi(\theta-t)\cot(\frac{t}{2})dt & =2\phi(\theta+\varepsilon)\ln(\vert\sin(\frac{\varepsilon}{2})\vert)+2\int_{-\pi}^{-\varepsilon}\phi'(\theta-t)\ln(\vert\sin(\frac{t}{2})\vert)dt\\
 & =2\phi(\theta+\varepsilon)\ln(\vert\sin(\frac{\varepsilon}{2})\vert)+2\int_{\varepsilon}^{\pi}\phi'(\theta+t)\ln(\vert\sin(\frac{t}{2})\vert)dt,
\end{split}
\label{eq3}
\end{equation}
By substituting \eqref{eq2} and \eqref{eq3} into \eqref{eq1}, the
Hilbert transform becomes 
\begin{equation}
\begin{alignedat}{1}H\{\phi\}(\theta) & =\frac{1}{2\pi}\lim_{\varepsilon\to0^{+}}\left\{ 2\left(\phi(\theta+\varepsilon)-\phi(\theta-\varepsilon)\right)\ln(\vert\sin(\frac{\varepsilon}{2})\vert)\right.\\
 & \left.+2\int_{\varepsilon}^{\pi}\left(\phi'(\theta+t)+\phi'(\theta-t)\right)\ln(\vert\sin(\frac{t}{2})\vert)dt\right\} 
\end{alignedat}
\label{eq:4}
\end{equation}
\end{proof}
Now, since $\phi$ is differentiable at $\theta$ and $\ln(\sin(\frac{\varepsilon}{2}))\underset{0^{+}}{\sim}\ln(\frac{\varepsilon}{2})$,
the first limit in \eqref{eq:4} becomes 
\[
\begin{split}\lim_{\varepsilon\to0^{+}}\left(\phi(\theta+\varepsilon)-\phi(\theta-\varepsilon)\right)\ln(\sin(\frac{\varepsilon}{2})) & =\lim_{\varepsilon\to0^{+}}\frac{\phi(\theta+\varepsilon)-\phi(\theta-\varepsilon)}{\varepsilon}\varepsilon\ln(\sin(\frac{\varepsilon}{2}))\\
 & =\lim_{\varepsilon\to0^{+}}2\phi'(\theta)\varepsilon\ln(\frac{\varepsilon}{2})\\
 & =0.
\end{split}
\]
For the second limit in \eqref{eq:4}, observe that 
\[
\left(\phi'(\theta+t)+\phi'(\theta-t)\right)\ln(\sin(\frac{t}{2}))\underset{0^{+}}{\sim}2\phi'(\theta)\ln(\frac{t}{2}).
\]
On the other hand, the function $t\mapsto\ln(\frac{t}{2})$ is integrable
on $[0,\pi].$ Then 
\[
\lim_{\varepsilon\to0^{+}}\int_{\varepsilon}^{\pi}\left(\phi'(\theta+t)+\phi'(\theta-t)\right)\ln(\sin(\frac{t}{2}))dt
\]
is finite. Therefore, 
\[
H\{\phi\}(\theta)=\int_{0}^{\pi}\left(\phi'(\theta+t)+\phi'(\theta-t)\right)\ln(\sin(\frac{t}{2}))dt
\]
 is finite. 
\begin{rem}
The proof of Theorem \ref{bounded deriv theorem} shows that $H\{\phi\}$
is continuous as it is the convolution between an $L^{\infty}$ function
and an $L^{1}$ function over $(0,\pi)$. 

The case where $\inf_{x\in(a,b)}\rho(x)=0,$i.e $\phi'(c)=\infty$
at some point $c\in[0,\pi],$ is inconclusive. The following two examples
illustrate this fact. The first example generates a bounded $\mu$-domain
while the second example produces an unbounded one.
\end{rem}

\begin{example}
Let $\mu$ be the probability distribution given by the density 
\[
\rho(x)=\frac{\alpha+1}{2}\vert x\vert^{\alpha}\chi_{(-1,1)}
\]
with $\alpha$ being a non-negative parameter. The c.d.f of $\mu$
is 
\[
F(x)=\frac{1}{2}(1-\vert x\vert^{1+\alpha})\chi_{(-1,0)}+\frac{1}{2}+\frac{1}{2}\vert x\vert^{1+\alpha}\chi_{(0,1)}
\]
and thus 
\[
\phi(\theta)=-(1-2\frac{\vert\theta\vert}{\pi})^{\frac{1}{\alpha+1}}\chi(\vert\theta\vert)_{(0,\frac{\pi}{2})}+2^{\frac{1}{\alpha+1}}(\frac{\vert\theta\vert}{\pi}-\frac{1}{2})^{\frac{1}{\alpha+1}}\chi(\vert\theta\vert)_{(\frac{\pi}{2},\pi)}.
\]

\begin{figure}
\begin{centering}
\includegraphics[width=8cm,height=8cm,keepaspectratio]{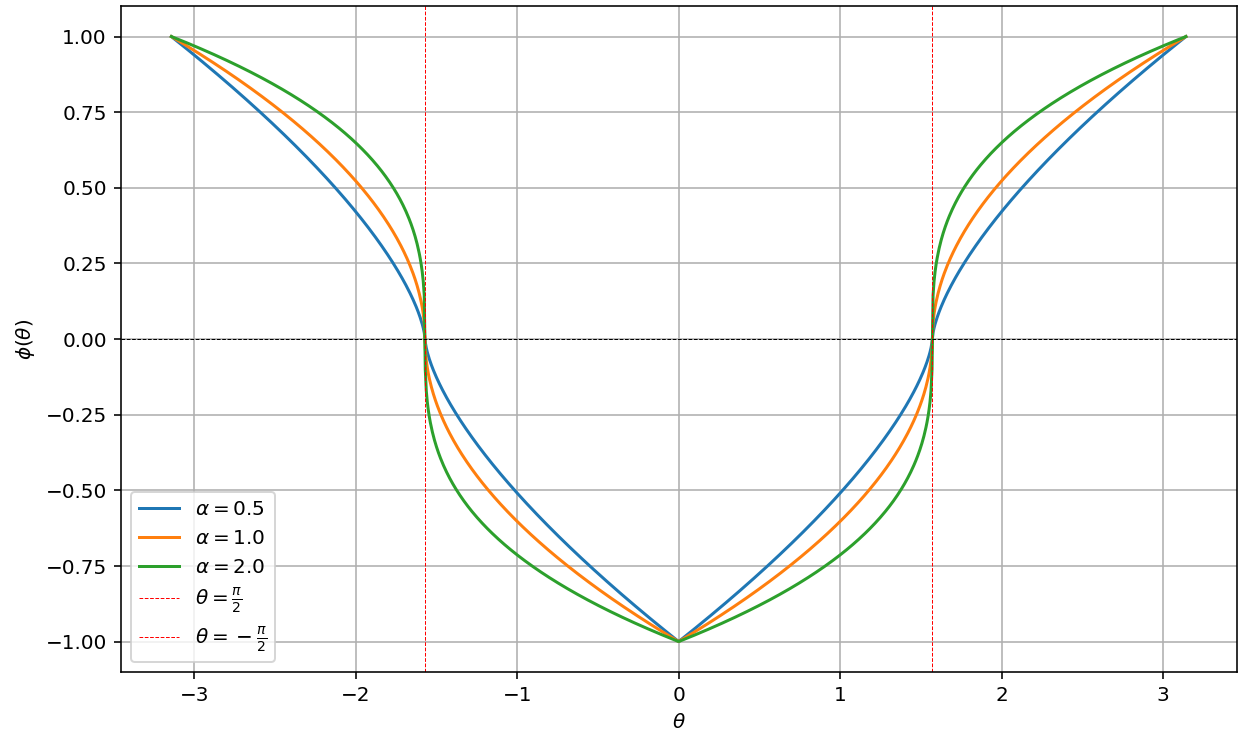} 
\par\end{centering}
\caption{Note that $\phi$ has vertical tangents at $\pm\frac{\pi}{2}$ as
$\phi'(\pm\frac{\pi}{2})=\pm\infty.$}
\end{figure}
\end{example}

Now, as $\cot(t)\underset{0}{\sim}\frac{1}{t}$ then we have the approximation:
\begin{equation}
\phi(t)\cot\left(\frac{\frac{\pi}{2}-t}{2}\right)\underset{\frac{\pi}{2}}{\sim}\frac{2\phi(t)}{\frac{\pi}{2}-t}=-\frac{(\frac{2}{\pi})^{\frac{1}{\alpha+1}}}{\vert t-\frac{\pi}{2}\vert{}^{\frac{\alpha}{\alpha+1}}}.\label{eq:approximation}
\end{equation}
As the R.H.S of \ref{eq:approximation} is integrable around $t=\frac{\pi}{2}$,
the function $\phi(t)\cot\left(\frac{\frac{\pi}{2}-t}{2}\right)$
is also integrable. The case $-\frac{\pi}{2}$ is similar. Hence 
\[
H\{\phi\}(\frac{\pi}{2})=\int_{-\pi}^{\pi}\phi(t)\cot\left(\frac{\frac{\pi}{2}-t}{2}\right)dt
\]
exists and is finite for all $\theta$. 
\begin{figure}
\centering{}\includegraphics[width=7cm,height=7cm,keepaspectratio]{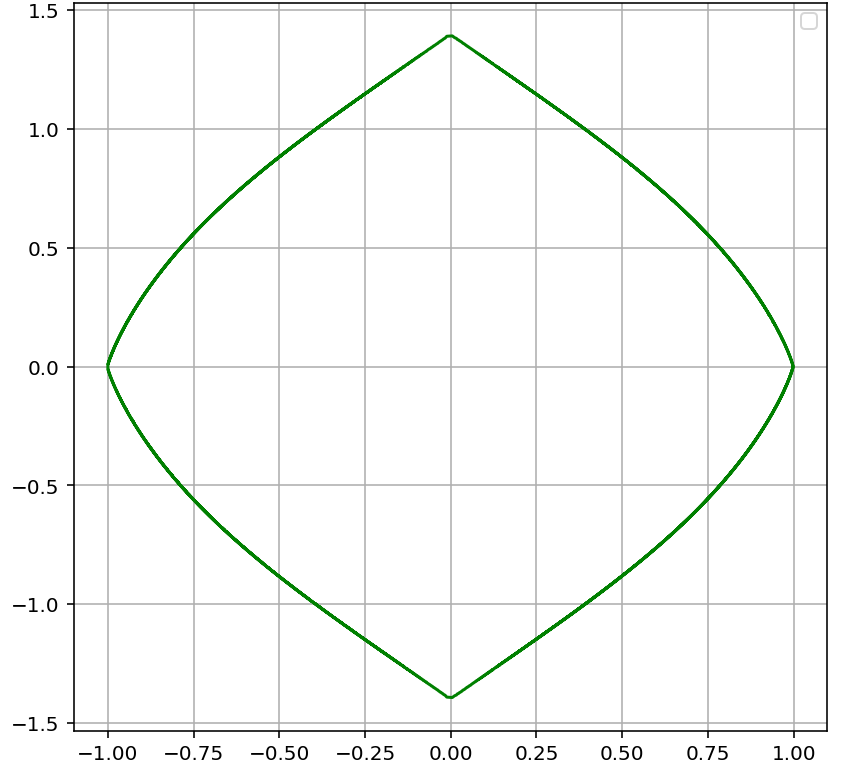}
\caption{The $\mu$-domain is bounded ( $\alpha=1$). }
\label{fig:enter-label} 
\end{figure}

\begin{example}
Before giving the example, we shall first provide the motivation behind.
Theorem \ref{bounded deriv theorem} says that any $\mu$-domain is
necessarily unbounded if $\phi'$ is bounded. That is, if we want
to seek a continuous distribution $\mu$ supported on an interval
$(\alpha,\beta)$ generating an unbounded domain, then necessarily
its pdf $\rho$ must hit the $x$-axis at some point $x_{0}\in(\alpha,\beta)$.
However hitting the value zero by $\rho$ is not enough as shown by
the previous example. Even more, the previous example shows that if
$f(x)\underset{x_{0}}{\sim}\vert x-x_{0}\vert^{\alpha}$ won't do
the job for any $\alpha>0$. So in order to boost the chance of getting
an unbounded domain, $\rho$ must be too much flat around $x_{0}$,
i.e its graph looks like it is overlapped with the $x$-axis at $x_{0}$.
In other words, we need $\rho$ such that
\end{example}

\begin{equation}
\frac{\rho(x)}{\vert x-x_{0}\vert{}^{\alpha}}\underset{x_{0}}{\rightarrow}0\label{eq: flat}
\end{equation}

for any positive $\alpha$. Inspired by this analysis, we shall show
that $\rho(x)=\kappa e^{-\frac{1}{\vert x\vert}}\mathbf{\chi}_{(-1,1)}$
is a suitable candidate to generate an unbounded domain ($\kappa$
being the normalization constant).

\begin{figure}
\begin{centering}
\includegraphics[width=8cm,height=8cm,keepaspectratio]{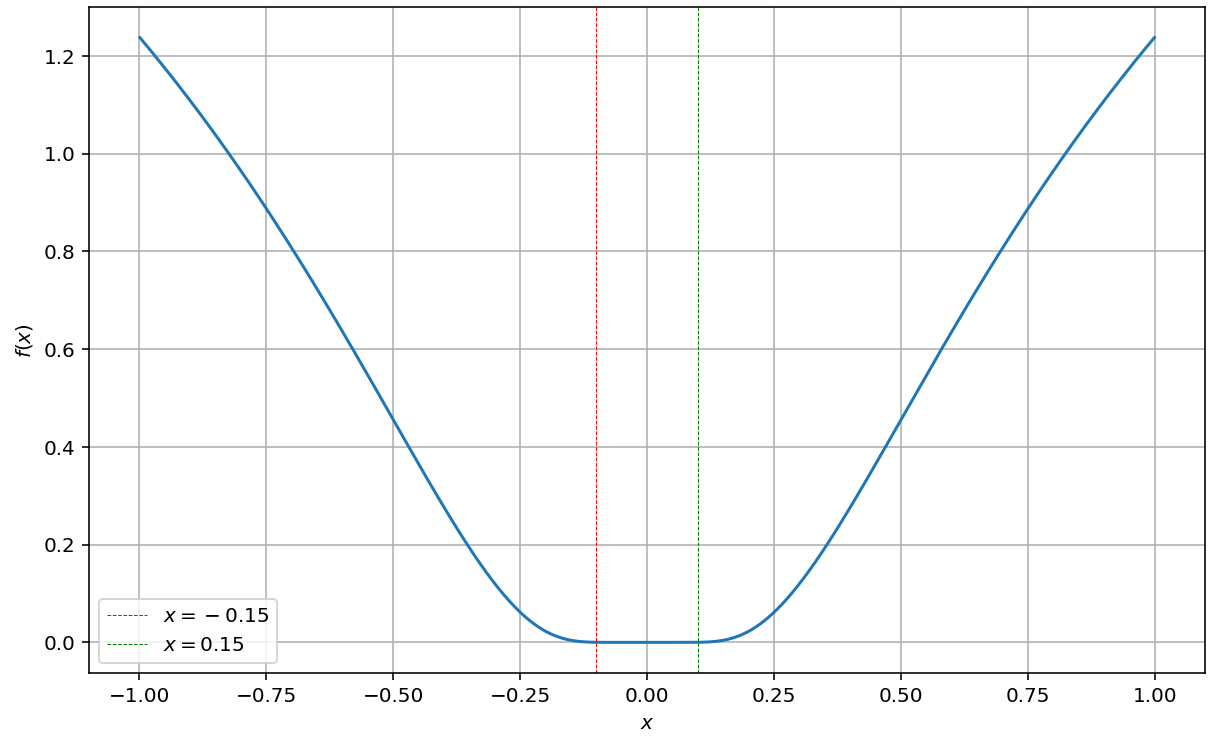} 
\par\end{centering}
\caption{Notice that the p.d.f $\rho(x)=\kappa e^{-\frac{1}{\vert x\vert}}$
is almost zero when $x\in(-0.15,0.15)$. That is, it looks like the
support of $\mu$ has a gap from $-0.15$ to $0.15$, which leads
to the unboundedness of corresponding $\mu$-domain. }
\end{figure}

Since $\rho$ is symmetric, the associated cumulative distribution
function takes the following form: 
\[
F(x)=\begin{cases}
\frac{1}{2}+\int_{0}^{x}ke^{-\frac{1}{t}}dt & x\in(0,1)\\
\frac{1}{2}-\int_{x}^{0}ke^{\frac{1}{t}}dt & x\in(-1,0)
\end{cases}.
\]
We have 
\[
\frac{1}{2}-\kappa e^{\frac{1}{x}}\underset{-1<x<0}{\leq}F(x)\underset{0<x<1}{\leq}\frac{1}{2}+\kappa e^{-\frac{1}{x}}.
\]
An elementary property of inverses infers that 
\[
\frac{1}{\ln(\frac{\frac{1}{2}-u}{\kappa})}\underset{0\leq u\leq\frac{1}{2}}{\geq}q(u)\underset{\frac{1}{2}\leq u\leq1}{\geq}-\frac{1}{\ln(\frac{u-\frac{1}{2}}{\kappa})}.
\]
Hence
\begin{equation}
\phi(t)\cot(\frac{\frac{\pi}{2}-t}{2})\geq-\frac{\cot(\frac{\vert\frac{\pi}{2}-t\vert}{2})}{\ln(\frac{\pi}{\kappa}\vert t-\frac{\pi}{2}\vert)}.\label{eq:inequality}
\end{equation}
The R.H.S of \ref{eq:inequality} is not integrable around $\frac{\pi}{2}$.
Then we deduce that $H\{\phi\}(\frac{\pi}{2})$ blows up, which infers
that $H\{\phi\}$ is unbounded.

\begin{figure}[H]
\begin{centering}
\includegraphics[width=8cm,height=8cm,keepaspectratio]{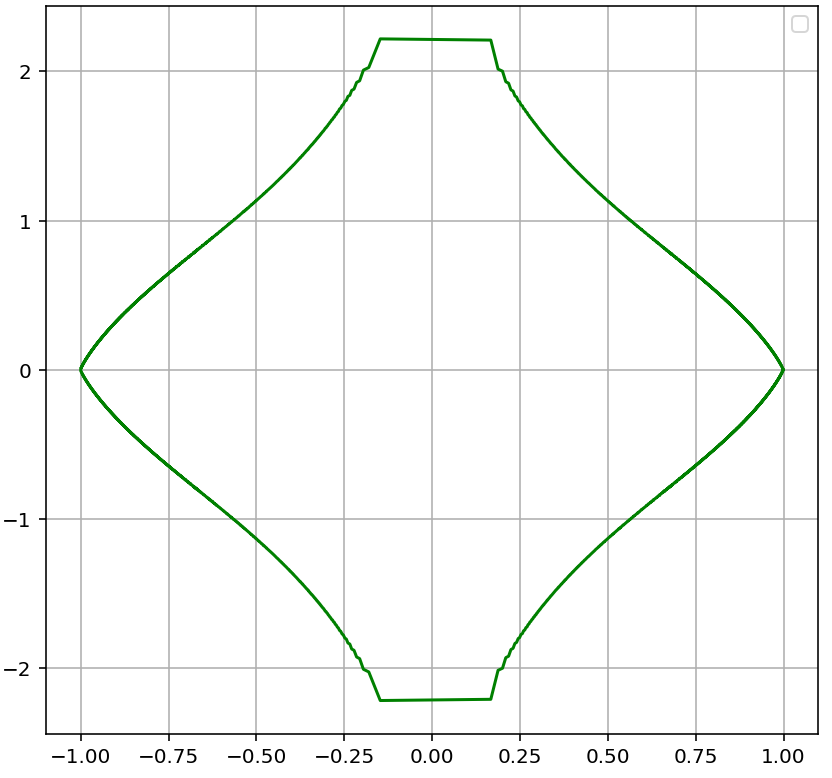}\caption{The $\mu$-domain generated by the probability distribution of density
$\rho(x)=\kappa e^{-\frac{1}{\vert x\vert}}$. Note that the boundary
keeps going vertically to $\pm\infty.$ However, as $e^{-\frac{1}{\vert x\vert}}$
is practically zero near the origin, and hence the \textquotedbl floating-point\textquotedbl{}
problem occurred. }
\par\end{centering}
\end{figure}

\section{Comments}

In this work, we have investigated the problem of the boundedness
of the $\mu$-domains and found some sufficient conditions on the
distribution $\mu$ to generate a bounded domain. In summary, in order
to have a blow-up at some point $x_{0}$, the graph of the quantile
function must be too much steep. This includes the case of support
with a gap. Assume that the support is $(-2,-1)\cup(1,2)$ for example,
the quantile function will have a jump at the point $u=\mathbb{P}\{(-2,-1)\}$.
At this point, the derivative is the Dirac function, which is the
most steep function ever. This explains the unboundedness of the corresponding
$\mu$-domain. An interesting question would be to discuss the necessity
of such conditions, namely the flatness of the p.d.f, i.e can one
find a distribution whose p.d.f satisfies \ref{eq: flat} with a bounded
$\mu$-domai.  \bibliographystyle{plain}
\bibliography{Maher}

\begin{thebibliography}{1}

\bibitem{Boudabra2020}
M.~Boudabra and G.~Markowsky.
\newblock A new solution to the conformal {S}korokhod embedding problem and applications to the {D}irichlet eigenvalue problem.
\newblock {\em Journal of Mathematical Analysis and Applications}, 491(2):124351, 2020.

\bibitem{boudabra2019remarks}
M.~Boudabra and G.~Markowsky.
\newblock Remarks on {G}ross' technique for obtaining a conformal {S}korohod embedding of planar {B}rownian motion.
\newblock {\em Electronic Communications in Probability}, 2020.

\bibitem{butzer1971hilbert}
P.~Butzer and R.~Nessel.
\newblock Hilbert transforms of periodic functions.
\newblock In {\em Fourier Analysis and Approximation}, pages 334--354. Springer, 1971.

\bibitem{duren2000theory}
P.~L Duren.
\newblock {\em Theory of $H^p$ spaces}.
\newblock Courier Corporation, 2000.

\bibitem{gross2019}
R.~Gross.
\newblock A conformal {S}korokhod embedding.
\newblock {\em Electronic Communications in Probability}, 2019.

\bibitem{king2009hilbert}
F.~King.
\newblock {\em Hilbert transforms}.
\newblock Cambridge University Press Cambridge, 2009.

\bibitem{mcgovern1980hilbert}
J~D. McGovern.
\newblock {\em The Hilbert Transform}.
\newblock PhD thesis, 1980.

\bibitem{morters2010brownian}
P.~M{\"o}rters and Y.~Peres.
\newblock {\em Brownian motion}, volume~30.
\newblock Cambridge University Press, 2010.

\bibitem{Obloj2004}
J.~Ob\l\'{o}j.
\newblock The {S}korokhod embedding problem and its offspring.
\newblock {\em Probability Surveys}, pages 321 -- 392, 2004.

\end{thebibliography}

\end{document}